\documentclass[12pt,a4paper]{article}
\usepackage[latin1]{inputenc}
\usepackage[english]{babel}
\usepackage{amsmath}
\usepackage{amsthm}
\usepackage{amsfonts}
\usepackage{amssymb}
\usepackage{graphicx}
\usepackage{dsfont}  

\newtheorem{theorem}{Theorem}[section]

\newtheorem{prop}[theorem]{Proposition}

\newtheorem {corollary}[theorem]{Corollary}
\newtheorem {lemma}[theorem]{Lemma}

\theoremstyle{definition}
\newtheorem {example}[theorem]{Example}

\theoremstyle{remark}
\newtheorem {rem}[theorem]{Remark}

\numberwithin{equation}{section}

\newcommand{\mr}{{\mathbb R}}

\newcommand{\mn}{{\mathbb N}}

\newcommand{\mc}{{\mathbb C}}
\newcommand{\md}{{\mathbb D}}

\renewcommand{\epsilon}{\varepsilon}
\newcommand{\eps}{\epsilon}

\newcommand{\ran}{\operatorname{ran}}
\newcommand{\hil}{\mathcal{H}}
\newcommand{\bdd}{\mathcal{B}}

\newcommand{\rank}{\operatorname{rank}}

\begin{document}
\title{Estimating the number of eigenvalues of linear operators on Banach spaces}
\author{M. Demuth\footnotemark[1], F. Hanauska\footnotemark[1], M. Hansmann\footnotemark[2], G. Katriel\footnotemark[3]}
\date{}
\maketitle

\renewcommand{\thefootnote}{\fnsymbol{footnote}}
\footnotetext[1]{Institute of Mathematics, Technical University of
Clausthal, Clausthal-Zellerfeld, Germany.}
\footnotetext[2]{Faculty of Mathematics, Chemnitz University of Technology,
Chemnitz, Germany.}
\footnotetext[3]{Department of Mathematics, ORT Braude College, Karmiel, Israel.}

\begin{abstract}
Let $L_0$ be a bounded operator on a Banach space, and consider a perturbation $L=L_0+K$, where $K$ is compact. This work is concerned with obtaining bounds
on the number of eigenvalues of $L$ in subsets of the complement of the essential spectrum of $L_0$, in terms of the approximation numbers of the perturbing operator $K$.
Our results can be considered as wide generalizations of classical results on the distribution of eigenvalues of compact operators, which correspond to the case $L_0=0$.
They also extend previous results on operators in Hilbert space. Our method employs complex analysis and a new finite-dimensional reduction, allowing us  
to avoid using the existing theory of determinants in Banach spaces, which would require strong restrictions on $K$. Several open questions regarding the sharpness of our
results are raised, and an example is constructed showing that there are some essential differences in the possible distribution of eigenvalues of operators in general 
Banach spaces, compared to the Hilbert space case. 
\end{abstract}

\section{Introduction}

The study of the distribution of eigenvalues of compact operators on a Banach space is a classical and well-developed subject (see, e.g., the monographs \cite{koenig} and \cite{pietsch}). Of primary concern is the problem of relating summability properties of some sequence of singular numbers (like the approximation numbers or Weyl-numbers) of a compact operator $L$ to the summability properties of its sequence of eigenvalues. For instance, a result of K\"onig \cite{MR0482266} (see also \cite{koenig}, Theorem 2.a.6), which generalizes the classical Weyl estimate for Hilbert space operators, says that 
\begin{equation*} 
 \sum_{j} |\lambda_j(L)|^p \leq 2(2e)^{p/2} \sum_{j} \alpha_j^p(L), \quad p>0,
\end{equation*}
where $\lambda_j(L)$ and $\alpha_j(L)$ denote the non-zero eigenvalues and the approximation numbers of $L$, respectively. An immediate consequence of this estimate is a bound on the number of eigenvalues $n_L(s)$  of $L$ {\it{outside}} the closed disk $B_s=\{\lambda\in \mc\;|\; |\lambda|\leq s\}$, namely 
\begin{equation}
  \label{eq:12}
n_L(s) \leq  \frac{2(2e)^{p/2}}{s^p}  \sum_{j} \alpha_j^p(L).  
\end{equation}

Our goal in the present paper is to prove bounds analogous to (\ref{eq:12}) for \emph{non-compact} operators $L=L_0+K$, where $L_0$ is a bounded operator and $K$ is a compact operator on a complex Banach space $X$. In such a case, Weyl's Theorem on preservation of the essential spectrum implies that  
$$\sigma_{ess}(L)=\sigma_{ess}(L_0)\subset \sigma(L_0)\subset B_{\|L_0\|},$$
so that, for any $s>\|L_0\|$, the part of the spectrum of $L_0$ outside $B_s$ consists of   
a finite number of eigenvalues of finite algebraic multiplicity. We wish to express this fact quantitatively by explicitly bounding the number of eigenvalues $n_L(s)$ in $B_s^c$. 
For example, one of our results is 
\begin{equation}\label{be}n_L(s) \leq C(p)\cdot  \frac{s}{(s-\|L_0\|)^{p+1}}  \sum_{j}\alpha_j^p(K) ,\qquad s>\|L_0\|,
\end{equation} 
see Corollary \ref{corrr}.  Note that in the very special case $L_0=0$ (so that $L=K$), (\ref{be}) reduces to the classical result (\ref{eq:12}), up to the value of a multiplicative constant.

While, as mentioned above, the distribution of eigenvalues of compact operators on Banach spaces is very well-studied, the same cannot be said for the type of generalization considered here. Indeed, essentially all results that we are aware of only concern the case where $L=L_0+K$ is a  Hilbert space operator - see, e.g., \cite{gohberg2, MR925418, MR2559715, demuth,  MR3054310, MR2463978} and references therein (This list is certainly quite incomplete. We only mention the classics and some more recent works. In particular, we only cite works about general non-selfadjoint operators). As we discuss below, the methods employed in the Hilbert space setting cannot be directly extended to Banach spaces, and therefore
some essentially new ideas are required, and these are developed here.

One of the key ideas used in \cite{MR2559715, demuth} is to identify the eigenvalues of $L$ with the zeros of a holomorphic function and then to use tools from complex analysis to obtain bounds on these zeros. This holomorphic function is defined in terms of some generalized determinants. 
In all cases, these generalized determinants can only be constructed given some summability assumptions on the approximation numbers of $K$. In the Banach space setting this method had been used in \cite{MR0298460} to study compact operators.

While we will pursue the same approach as mentioned in the previous paragraph, a key technical innovation of the present work is that we will {\it{not}} rely on the known determinant theory for Banach space operators (as developed, e.g., in \cite{pietsch} and \cite{MR568991}). Instead, we will use a finite-dimensional reduction argument to construct the required
holomorphic function, whose zeros in a certain domain $\Omega \subset \mc$ coincide with the eigenvalues of $L$ in this domain, using only (generalized) determinants of finite-rank operators. In this way we are able to avoid the strong assumptions on $K$ required for directly employing 
infinite-dimensional determinant theory. This enables us to obtain results in which the only assumption on $K$ is that it is approximable by finite rank operators, 
i.e. that its approximation numbers $\alpha_j(K)$ tend to zero (but not assuming anything about their summability). In particular, this means that our results are new even when specializing to the case of Hilbert space operators. 
When the approximation numbers are ($p$-)summable, our bounds take a particularly simple form, as in (\ref{be}) above.

The plan of this paper is as follows: In the next section we will gather some preliminary results concerning approximation numbers and determinants of finite rank operators. In Section 3 we will construct a holomorphic function whose zeros coincide with the eigenvalues of $L$, and in Section 4  we will prove our eigenvalue estimates. In the final Section 5 we will provide some remarks concerning the sharpness of our results, including a comparison with previously obtained results in Hilbert spaces and an example which shows 
that there is an essential difference in the distribution of eigenvalues between the Hilbert space case and the general Banach space case considered here. 

\section{Preliminaries} 
\subsection{Approximation numbers}

Let $(X,\|.\|_X)$ be a complex Banach space and let $\bdd(X)$ and $\mathcal{F}(X)$ denote the classes of bounded and finite rank operators on $X$, respectively. The operator norm of $L \in \bdd(X)$ will be denoted by $\|L\|$. We define the \emph{$n$th approximation number} of $L \in \bdd(X)$ as
\begin{align*} 
\alpha_n(L):=\inf\{\|L-F\|: F \in \mathcal{F}(X), \rank(F)< n\}, \quad n \in \mn.
\end{align*}
If $\alpha_n(L) \to 0$ for $n \to \infty$, then $L$ is a compact operator on $X$. On the other hand, in some Banach spaces not every compact operator can be approximated by finite rank operators, as has been shown by Enflo \cite{perenflo}.  
 
We recall the following properties of the approximation numbers (see, e.g., \cite{koenig} p. 69): For $K,L,M \in \bdd(X)$ and $n,m \in \mn$
\begin{enumerate}
\item[(i)] $\|L\|=\alpha_1(L)\geq \alpha_2(L) \geq \dots \geq 0$, 
\item[(ii)] $\alpha_{n+m-1}(K+L)\leq \alpha_n(K)+\alpha_m(L)$, 
\item[(iii)] $\alpha_{n}(KLM)\leq \|K\|\alpha_n(L)\|M\|$,
\item[(iv)] $\alpha_n(L)=0$ if $\rank(L) < n$. 
\end{enumerate}

Let $K \in \bdd(X)$ be a compact operator and let $\lambda_1(K),\lambda_2(K), \ldots$ denote its non-zero eigenvalues, ordered such that $|\lambda_1(K)| \geq |\lambda_2(K)| \geq \ldots>0$ and counted according to their algebraic multiplicity, where the algebraic multiplicity is defined as the rank of the Riesz projection of $K$ (see, e.g., \cite{gohberg1}) with respect to the considered eigenvalue.  The following estimate is due to K\"onig (see \cite{koenig}, Theorem 2.a.6): For $p \in (0,\infty)$ 
\begin{equation}
  \label{eq:7}
\sum_{j}|\lambda_j(K)|^p \leq 2(2e)^{p/2}\sum_{j}\alpha_j^p(K).  
\end{equation}

\subsection{Determinants of finite rank operators}

Let $F \in \mathcal{F}(X)$. For $n \in \mn$ the \emph{$n$-regularized determinant}  of $\mathds{1}-F$, where $\mathds{1}$ denotes the identity operator on $X$, is defined in terms of the (finite number of) eigenvalues of $F$, as follows:
\begin{eqnarray*}
  {\det}_n(\mathds{1}-F):= 
      \prod_{k} \left[ (1-\lambda_k(F))\exp\left( \sum_{j=1}^{n-1} \frac{\lambda_k^j(F)}{j} \right) \right].
\end{eqnarray*}
Here we use the standard convention that $\sum_{j=1}^{0}(\ldots):=0$. As references for regularized determinants we refer to \cite{MR1009163},\cite{gohberg2} and \cite{simon}, for the case of operators on a Hilbert space, and to \cite{MR1744872} for the general case. 

As a first simple but important property of regularized determinants let us note that  ${\det}_n(\mathds{1}-F) \neq 0$ iff $\mathds{1}-F$ is invertible in $\bdd(X)$. In the following, we will gather some less obvious properties. To this end, let us denote the extended complex plane by $\hat \mc$ and for a subspace $Y$ of $X$ let us set 
 $$\mathcal{F}(X;Y) := \{ F \in \mathcal{F}(X): \ran(F) \subset Y\}.$$
We note that $Y$ is an invariant subspace of $F \in \mathcal{F}(X;Y)$ and that the non-zero eigenvalues of $F$ and $F_Y$ (the restriction of $F$ to $Y$) coincide. In particular,
$$ {\det}_n(\mathds{1}-F)={\det}_n(\mathds{1}_Y-F_Y).$$
\begin{prop}\label{prop01} 
Let $G \subset \hat \mc$ be open and let $Y \subset X$ be a finite-dimensional subspace. Suppose that  $F(\lambda) \in \mathcal{F}(X;Y)$ for all $\lambda \in G$. Then the following holds: If $\lambda \mapsto F(\lambda) $ is analytic on $G$, then $\lambda \mapsto \det_n(\mathds{1}-F(\lambda))$ is analytic on $G$ as well. 
\end{prop} 
\begin{proof}
We would like to use the fact that for Hilbert space operators the analyticity of the regularized determinant has been proven in \cite{MR0482328}. To this end, for every $\lambda \in G$ we denote by $F_Y(\lambda)$ the restriction of $F(\lambda)$ to $Y$. From the discussion preceeding the proposition we know that 
$ {\det}_n(\mathds{1}-F(\lambda))={\det}_n(\mathds{1}_Y-F_Y(\lambda))$. 
Choose a norm $\|.\|_Y$ on $Y$ such that $\hil=(Y,\|.\|_Y)$ is a Hilbert space and let $J : \hil \to (Y,\|.\|_X)$ denote the canonical isomorphism. Since the eigenvalues of $F_Y(\lambda)$ and $J^{-1}F_Y(\lambda)J$ coincide (including multiplicity), we obtain that $ {\det}_n(\mathds{1}_Y-F_Y(\lambda))= {\det}_n(\mathds{1}_\hil-J^{-1}F_Y(\lambda)J)$. It remains to note that $J^{-1}F_Y(\lambda)J \in \mathcal{F}(\hil)$ is analytic and of finite rank.

\end{proof} 
\begin{rem}
  The assumption that the ranges of all operators $F(\lambda)$ are contained in a single space $Y$ is certainly not necessary. However, it is sufficient for our purposes and, as we have seen above, it allows for a very easy proof. For completeness we should note that, without this assumption, the analyticity of $\lambda \mapsto {\det_1}(\mathds{1}-F(\lambda))$ has been shown in \cite{MR0417827}. 
\end{rem} 
\begin{prop}\label{prop02}
Let $p \in (0,\infty)$ and $F \in \mathcal{F}(X)$. Then there exists a constant $\Gamma_p$, depending only on $p$, such that 
\begin{equation}
  \label{eq:2}
  |{\det}_{\lceil p \rceil} (\mathds{1}-F)| \leq \exp\left( 2(2e)^{p/2} \Gamma_p \sum_{j} \alpha_j^p(F) \right),
\end{equation}
where $\lceil p \rceil = \min\{ n \in \mn: n \geq p\}$. 
\end{prop}
\begin{rem} 
 For upper and lower bounds on $\Gamma_p$ we refer to  \cite{MR2391269}. 
\end{rem}
\begin{proof}
There exists a constant $\Gamma_p>0$ such that for  $\lambda \in \mc$: 
$$ \left|(1-\lambda)\exp\left( \sum_{j=1}^{\lceil p \rceil-1} \frac{\lambda^j}{j} \right)\right| \leq \exp(\Gamma_p |\lambda|^p),$$
see \cite{MR1009163}, p.1107. This implies that 
$$   |{\det}_{\lceil p \rceil} (\mathds{1}-F)| \leq \exp\left( \Gamma_p \sum_{j} |\lambda_j(F)|^p \right).$$
Now apply estimate (\ref{eq:7}).
\end{proof}
We recall that the \emph{essential spectrum} of $L \in \bdd(X)$ is defined as 
\begin{equation}
  \label{eq:4}
  \sigma_{ess}(L)=\{ \lambda \in \mc : \lambda - L \text{ is not a Fredholm operator } \}.
\end{equation}
Here an operator is called Fredholm if it has closed range and both its kernel and cokernel are finite-dimensional.
Moreover, the \emph{discrete spectrum} of $L$, $\sigma_d(L)$, consists of all isolated eigenvalues of $L$ of finite algebraic multiplicity. The elements of the discrete spectrum will be called discrete eigenvalues. They can accumulate only at the essential spectrum.

By Weyl's theorem, the essential spectrum is invariant under a compact perturbation, so if $L_1 \in \bdd(X)$ and $L_2-L_1$ is compact, then $\sigma_{ess}(L_2)=\sigma_{ess}(L_1)$. The discrete spectra of $L_2$ and $L_1$ certainly need not coincide. In the following, assuming that the difference $L_2-L_1$ is of finite rank, we will identify the discrete eigenvalues of $L_2$ outside the spectrum of $L_1$ with the zeros of a certain holomorphic function. 
\begin{prop}\label{prop03}
Let $L_1,L_2 \in \bdd(X)$ and suppose that $L_2-L_1 \in \mathcal{F}(X)$. Let $U$ denote the unbounded component of $\mc \setminus \sigma(L_1)$. For $\lambda \in U$ and $p \in (0,\infty)$ define 
\begin{equation}
  \label{eq:6}
  d_{p}^{L_2,L_1}(\lambda):= {\det}_{\lceil p \rceil}(\mathds{1}-(L_2-L_1)(\lambda-L_1)^{-1}).
\end{equation}
Then the following hold: 
\begin{enumerate}
    \item[(i)] $d_{p}^{L_2,L_1}$ is analytic on $U$,    
   \item[(ii)] $ \log |d_{p}^{L_2,L_1}(\lambda)| \leq  2(2e)^{p/2} \Gamma_p \sum_k \alpha_k^p\big((L_2-L_1)(\lambda - L_1)^{-1}\big)$, where $\Gamma_p$ is as in Proposition \ref{prop02},
  \item[(iii)] $\lambda_0 \in U$ is a zero of $d_{p}^{L_2,L_1}$ of order $k$ if and only if it is a discrete eigenvalue of $L_2$ of algebraic multiplicity $k$.
\end{enumerate}
\end{prop}

Following the terminology in \cite{gohberg2} we call $d_{p}^{L_2,L_1}$ the \emph{$p$th perturbation determinant} of $L_2$ by $L_1$.
\begin{proof}
 For $\lambda \in U$, we set $F(\lambda):=(L_2-L_1)(\lambda-L_1)^{-1}$. Then $F(\lambda)$ is analytic and of finite rank and $\ran(F(\lambda)) \subset \ran(L_2-L_1)$ for all $\lambda \in U$. The first two statements now follow from Proposition \ref{prop01} and \ref{prop02}, respectively. The third statement is well-known for Hilbert space operators (see, e.g., \cite{MR2662459}, Theorem 21). For completeness, let us sketch the proof of the general case.
 
  First of all, it is clear that $\lambda_0 \in U$ is a zero of $d_{p}^{L_2,L_1}$ iff $1$ is an eigenvalue of $F(\lambda_0)$. We now show that this is the case iff $\lambda_0$ is an eigenvalue of $L_2$. Indeed, if $1$ is an eigenvalue of $F(\lambda_0)$ then there exists
 $x\in X$, $x\neq 0$, with $F(\lambda_0)x=x$, that is $(L_2-L_1)(\lambda_0-L_1)^{-1}x=x$, so setting $y=(\lambda_0-L_1)^{-1}x$ we have
 $(L_2-L_1)y=(\lambda_0-L_1)y$, that is $L_2y=\lambda_0 y$, so $\lambda_0$ is an eigenvalue of $L_2$. Conversely, if $\lambda_0\in \mc\setminus \sigma(L_1)$ is an eigenvalue of $L_2$, then we have  $L_2y=\lambda_0 y$ for some $y\in X$, $y\neq 0$.  Thus setting $x=(\lambda_0-L_1)y$ we obtain
 $$F(\lambda_0)x=(L_2-L_1)(\lambda_0-L_1)^{-1}x=(L_2-L_1)y=(\lambda_0-L_1)y=x,$$
 so that indeed $1$ is an eigenvalue of $F(\lambda_0)$.
 
That all eigenvalues of $L_2$ in $U$ are discrete follows from the fact that the spectrum of $L_2$ in the unbounded component of $\mc \setminus \sigma_{ess}(L_2)$ is purely discrete  (see \cite{b_Davies}, Theorem 4.3.18) and from the fact that 
$$U \; \subset \; \mc \setminus \sigma(L_1) \; \subset \; \mc \setminus \sigma_{ess}(L_1) \;=\; \mc \setminus \sigma_{ess}(L_2),$$
which shows that $U$ is a subset of this unbounded component.

 It remains to show that the multiplicities of $\lambda_0$ as a zero of $d_{p}^{L_2,L_1}$ and as an eigenvalue of $L_2$ coincide. For that purpose, let us first note that it is no restriction to assume that $\lambda_0 \neq 0$. Now we denote the Riesz-Projection (see, e.g., \cite{gohberg1}) of $L_2$ with respect to $\lambda_0 \in \sigma_d(L_2) \cap U$ by $P$, and we set $T=L_2P$ and $T^\perp=L_2(\mathds{1}-P)$. Note that $T$ is of finite rank, with $\sigma(T)=\{\lambda_0\}$, and that $\lambda_0 \notin \sigma(T^\perp)$. In particular, there exists a ball $B$ around $\lambda_0$ such that $0 \notin B$ and such that $\lambda-L_1$ and $\lambda-T^\perp$ are invertible for all $\lambda \in B$. Now a short computation, using $TT^\perp=T^\perp T=0$ and $L_2=T+T^\perp$, shows that for $\lambda \in B$
$$ \mathds{1}-(L_2-L_1)(\lambda-L_1)^{-1} = (\mathds{1}-\lambda^{-1}T)(\mathds{1}-(T^\perp-L_1)(\lambda-L_1)^{-1}).$$
Hence, following \cite{MR1744872}, p.202, there exists a holomorphic function $C_{L_2,L_1,p}(\lambda)$ such that 
$$ d_{p}^{L_2,L_1}(\lambda) = {\det}_{1}(\mathds{1}-\lambda^{-1}T){\det}_{\lceil p \rceil}(\mathds{1}-(T^\perp-L_1)(\lambda-L_1)^{-1})\exp(C_{L_2,L_1,p}(\lambda)).$$
The operator $\mathds{1}-(T^\perp-L_1)(\lambda-L_1)^{-1}=(\lambda-T^\perp)(\lambda-L_1)^{-1}$ is invertible for $\lambda \in B$, so we see that the multiplicity of $\lambda_0$ as a zero of $d_{p}^{L_2,L_1}$ coincides with its multiplicity as a zero of $\lambda \mapsto {\det}_{1}(\mathds{1}-\lambda^{-1}T) = (1-\lambda^{-1}\lambda_0)^{\rank(P)}$. 
But the rank of $P$ coincides with the algebraic multiplicity of $\lambda_0$ as an eigenvalue of $L_2$.
\end{proof}

\section{Eigenvalues as zeros of an analytic function} 

Let $L_0 \in \bdd(X)$ and let $K$ be a compact operator on $X$. We assume that $K$ is the uniform limit of finite rank operators, i.e. 
\[ \lim_{n\rightarrow \infty}\alpha_n(K)=0.\]
In the following, we will be interested in the discrete spectrum of the operator 
\[ L:=L_0+K.\]
Let $\Omega \subset \hat{\mc}$ denote a connected, open set with $\infty \in \Omega$ and such that 
\begin{align}
\overline{\Omega}\cap \sigma(L_0)=\emptyset,\label{eq1.1a}
\end{align}
which implies that 
\[ S(\Omega):=\sup_{\lambda\in \Omega}\|(\lambda-L_0)^{-1}\|< \infty.\]
\begin{rem}
We note that the resolvent $R: \Omega \to \bdd(X)$, $R(\lambda)= (\lambda-L_0)^{-1}$, is analytic on $\Omega$ with $R(\infty)=0$.
\end{rem} 
Our final aim (see Section 4) is to prove an upper bound on the number of discrete eigenvalues of $L$ in $\Omega$. As a first step, we are going to relate the discrete eigenvalues of $L$ to the zeros of a certain holomorphic function, which can be estimated from above in terms of the approximation numbers of $K$. 

\begin{theorem} \label{prop1} Let $p \in (0,\infty)$ and $N\in\mathbb{N}_0$ such that $\alpha_{N+1}(K)<{1}/{S(\Omega)}$. Then there exists a bounded holomorphic function $d:\Omega\rightarrow \mathbb{C}$ (depending on $p,N,L$ and $L_0$) with the following properties: 
  \begin{enumerate} 
      \item[(i)] $d(\infty)=1$,
      \item[(ii)] for all $\lambda \in \Omega$ we have 
    \begin{small}
      \begin{align*}
|d(\lambda)|&\leq  \exp\left( \frac{C_p \|(\lambda - L_0)^{-1}\|^p }{\left(1-\alpha_{N+1}(K)\|(\lambda-L_0)^{-1}\|\right)^p}\sum_{j=1}^N\Big( \alpha_{N+1}(K) + \alpha_j(K) \Big)^p\right)\\
&\leq \exp\left( \frac{C_pS(\Omega)^p }{(1-\alpha_{N+1}(K)S(\Omega))^p}\sum_{j=1}^N\Big( \alpha_{N+1}(K) + \alpha_j(K) \Big)^p \right),
\end{align*}
\end{small}
where
\begin{equation}
  \label{eq:9}
C_p= 2(2e)^{p/2} \Gamma_p,  
\end{equation}
with $\Gamma_p$ as in Proposition \ref{prop02},
      \item[(iii)] $\lambda_0 \in \Omega$ is a zero of $d$ of order $m$ if and only if it is a discrete eigenvalue of $L$ of algebraic multiplicity $m$.
  \end{enumerate}
\end{theorem} 
The proof of this theorem consists of two steps: First, we will use a finite dimensional reduction argument to construct a family of holomorphic functions  which satisfy point (i) and (iii) of the theorem (and which 'almost' satisfy estimate (ii)). In the second step, we will use an approximation argument involving Montel's theorem to construct the function $d$.  

To begin, let us fix $p> 0$ and let  $N\in\mathbb{N}_0$ be chosen such that $\alpha_{N+1}(K)< 1/{S(\Omega)}$. Then for $\eta \in (0,\frac{1}{S(\Omega)}-\alpha_{N+1}(K))$ there exists $F \in \mathcal{F}(X)$ of rank at most $N$ such that
\begin{align*}
 \|K-F\|<\alpha_{N+1}(K)+\eta
\end{align*}
and so for all $\lambda \in \Omega$ we can estimate
\begin{eqnarray*}
&& \|(K-F)(\lambda-L_0)^{-1}\| \leq \|K-F\|\|(\lambda-L_0)^{-1}\| \\
&\leq&  (\alpha_{N+1}(K)+\eta)\|(\lambda-L_0)^{-1}\| \leq (\alpha_{N+1}(K)+\eta)S(\Omega) 
< 1.  
\end{eqnarray*}
In particular, the operator $\mathds{1}-(K-F)(\lambda-L_0)^{-1}$ is invertible and
\begin{equation}
  \label{eq:5}
  \Big\| \big[\mathds{1}-(K-F)(\lambda-L_0)^{-1}\big]^{-1}\Big\| \leq \Big(1-(\alpha_{N+1}(K)+\eta)\|(\lambda-L_0)^{-1}\|\Big)^{-1}.
\end{equation}
Therefore, for $\lambda \in \Omega \setminus \{\infty\}$ the operator
\begin{equation}
  \label{eq:8}
  \lambda-(L-F)=[\mathds{1}-(K-F)(\lambda-L_0)^{-1}] (\lambda-L_0)
\end{equation}
is invertible, as the product of two invertible operators, and so $\Omega \subset \hat\mc \setminus \sigma(L-F)$. It follows that the perturbation determinant 
\begin{align}\label{defd}
 d_F(\lambda):= d_{p}^{L,L-F}(\lambda)={\det}_{\lceil p \rceil}(\mathds{1}-F[\lambda-(L-F)]^{-1})
\end{align}
is well-defined and analytic on $\Omega$, and we have $d_F(\infty)=1$. From Proposition  \ref{prop03} (with $L_2=L,L_1=L-F$) we further know that $\lambda_0 \in \Omega$ is a zero of $ d_{F}$ if and only if it is a discrete eigenvalue of $L$ of the same multiplicity. Proposition \ref{prop03} also implies that
\begin{equation*}
  |d_{F}(\lambda)| \leq \exp \left(2(2e)^{p/2} \Gamma_p \sum_{j=1}^N \alpha_j^p\Big(F[\lambda-(L-F)]^{-1}\Big) \right).
\end{equation*}
Let us estimate the approximation numbers on the right-hand side of the previous inequality: Using (\ref{eq:8}) and (\ref{eq:5}) we obtain
\begin{eqnarray*}
&& \alpha_j\Big(F[\lambda-(L-F)]^{-1}\Big)  \\
&=&  \alpha_j\Big(F(\lambda - L_0)^{-1}\big[\mathds{1}-(K-F)(\lambda-L_0)^{-1}\big]^{-1}\Big) \\
 &\leq& \alpha_j\Big(F(\lambda - L_0)^{-1}\Big)\Big\|\big[\mathds{1}-(K-F)(\lambda-L_0)^{-1}\big]^{-1}\Big\| \\
 &\leq& \frac{\alpha_j(F(\lambda - L_0)^{-1})}{1-(\alpha_{N+1}(K)+\eta)\|(\lambda-L_0)^{-1}\|}.  \end{eqnarray*}
We continue, using that $\alpha_j(A+B) \leq  \alpha_j(A) + \|B\|$,
\begin{eqnarray*} 
\alpha_j\Big(F(\lambda - L_0)^{-1}\Big)&=& \alpha_j\Big((F-K)(\lambda - L_0)^{-1}+ K(\lambda -L_0)^{-1}\Big)\\
 &\leq& \|(F-K)(\lambda - L_0)^{-1}\|+ \alpha_j\Big(K(\lambda -L_0)^{-1}\Big)\\
 &\leq&  \|(\lambda - L_0)^{-1}\| \left( \|F-K\| + \alpha_j(K) \right) \\
 &\leq&  \|(\lambda - L_0)^{-1}\| \left( \alpha_{N+1}(K)+\eta + \alpha_j(K) \right).
\end{eqnarray*}
Therefore
$$\alpha_j\Big(F[\lambda-(L-F)]^{-1}\Big)  \leq \frac{\|(\lambda - L_0)^{-1}\| \left( \alpha_{N+1}(K)+\eta + \alpha_j(K) \right)}{1-(\alpha_{N+1}(K)+\eta)\|(\lambda-L_0)^{-1}\|},$$
and so
\begin{small}
$$\sum_{j=1}^N\alpha_j^p\Big(F[\lambda-(L-F)]^{-1}\Big)  \leq \frac{\|(\lambda - L_0)^{-1}\|^p \sum_{j=1}^N\Big( \alpha_{N+1}(K)+\eta + \alpha_j(K) \Big)^p}{\left(1-(\alpha_{N+1}(K)+\eta)\|(\lambda-L_0)^{-1}\|\right)^p}.$$
\end{small}
Finally,
we obtain the following upper bound on the function $d_F$: For all $\lambda \in \Omega$ 
\begin{small}
\begin{eqnarray}\label{bnd}
|d_{F}(\lambda)| \leq \exp\left(2(2e)^{p/2} \Gamma_p \frac{\|(\lambda - L_0)^{-1}\|^p \sum_{j=1}^N\Big( \alpha_{N+1}(K)+\eta + \alpha_j(K) \Big)^p}{\left(1-(\alpha_{N+1}(K)+\eta)\|(\lambda-L_0)^{-1}\|\right)^p}\right).\nonumber\\
\end{eqnarray}
\end{small}
Let us collect all our results up to this point in the following lemma. 
\begin{lemma}\label{lemma1.1} 
Let $N\in\mathbb{N}_0$ be such that $\alpha_{N+1}(K)< 1 / S(\Omega)$ and fix some $ \eta \in (0,\frac{1}{S(\Omega)}-\alpha_{N+1}(K))$. Then there exists $F$ of rank at most $N$ such that the holomorphic function $d_{F}: \Omega\rightarrow \mathbb{C}$ defined by (\ref{defd}) satisfies (\ref{bnd}) and $d_{F}(\infty)=1$. In addition,  $\lambda_0 \in \Omega$ is a zero of $d_{F}$ of order $m$ if and only if it is a discrete eigenvalue of $L$ of algebraic multiplicity $m$.
\end{lemma} 
We conclude the proof of Theorem \ref{prop1} with the following limiting argument: Choose $N_0 \in \mn$ such that $\alpha_{N_0+1}(K)< 1 / S(\Omega)$. Let  $l_0\in\mathbb{N}$ denote the smallest integer such that $\frac{1}{l_0}<\frac{1}{S(\Omega)}-\alpha_{N_0+1}(K)$. Then by the previous lemma for every $l\geq l_0$ there exists an operator $F_l$ of rank at most $N_0$ such that the holomorphic function $d_{F_l}$ on $\Omega$ defined by (\ref{defd}) satisfies 
\begin{eqnarray*} 
&&  |d_{F_l}(\lambda)| \\
&\leq&  \exp\left(2(2e)^{p/2} \Gamma_p \frac{\|(\lambda - L_0)^{-1}\|^p \sum_{j=1}^N\Big( \alpha_{N+1}(K)+\frac{1}{l} + \alpha_j(K) \Big)^p}{\left(1-(\alpha_{N+1}(K)+\frac{1}{l})\|(\lambda-L_0)^{-1}\|\right)^p}\right) \\
&\leq&  \exp\left(2(2e)^{p/2} \Gamma_p \frac{S(\Omega)^p \sum_{j=1}^N\Big( \alpha_{N+1}(K)+\frac{1}{l_0} + \alpha_j(K) \Big)^p}{\left(1-(\alpha_{N+1}(K)+\frac{1}{l_0})S(\Omega)\right)^p}\right)
\end{eqnarray*}
for all $\lambda\in \Omega$. The right-hand side of this inequality is a uniform bound for the sequence of holomorphic functions $\big(d_{F_l}\big)_{l\geq l_0}$. Using Montel's Theorem (see e.g. \cite{rudin}, Theorem 14.6), there exists a locally uniformly convergent subsequence. Calling the local uniform limit of this subsequence $d$, let us check that this function satisfies all conditions of Theorem \ref{prop1}: First of all, it is clear that $d$ satisfies the estimate stated under point (ii). This uniform bound on $d$ also implies that $d$ is holomorphic at infinity. The local uniform convergence of $d_{F_l}$ and the fact that $d_{F_l}(\infty)=1$ imply that also $d(\infty)=1$. Finally, Hurwitz' theorem (see, e.g., \cite{MR503901}) and the fact that the zero sets of all functions $d_{F_l}$ coincide with $\sigma_d(L)\cap \Omega$ imply the assertion concerning the zero set of $d$. This completes the proof of Theorem \ref{prop1}. 
\begin{rem} 	
We note that if the approximation numbers $\{\alpha_j(K)\}$ are $p$-summable for some $p\in(0,\infty)$, then we can do without assumption (\ref{eq1.1a})  and prove that there exists a holomorphic function $\tilde{d}:  U \rightarrow \mathbb{C}$, defined on the \emph{entire} unbounded component $U$ of $\mc \setminus \sigma(L_0)$,  which satisfies points (i) and (iii) of Theorem \ref{prop1} and the inequality 
\begin{align*} 
\log |\tilde{d}(\lambda)|\leq C_p \|(L_0-\lambda)^{-1}\|^p\sum_{j=1}^\infty \alpha_j^p(K), \qquad \lambda \in U.
\end{align*} 
We are not going to use this result in the present paper, so let us just provide a rough sketch of proof: First, one approximates the set $U$ with sets $U_n$ which satisfy (\ref{eq1.1a}), then one applies Theorem \ref{prop1} to obtain a sequence of holomorphic functions $d_n$ defined on $U_n$ and finally one uses Hurwitz' theorem (twice) to obtain the desired function $\tilde{d}$. 
\end{rem}

\section{Estimating the number of eigenvalues}

We repeat our assumptions from the previous section: Let $L_0 \in \bdd(X)$ and let $K$ be a compact operator on $X$, which is the uniform limit of finite rank operators. Set $L:=L_0+K$. We wish to estimate the number of eigenvalues of $L$ in a domain $\Omega$ which is `away' from the spectrum of $L_0$. To quantify the notion of `away' we 
recall the definition of the $\epsilon$-pseudospectrum of a linear operator $L_0$: 
\begin{equation}
\sigma_{\eps}(L_0) := \{ \lambda \in \mc : \|(\lambda-L_0)^{-1}\| \geq \eps^{-1}\}. 
\end{equation}

In this section we will assume that
$\Omega \subset \hat\mc$ is an open and \emph{simply connected} set satisfying $\infty \in \Omega$, with 
\begin{align}\label{pss}
\overline{\Omega}\cap \sigma_\eps(L_0)=\emptyset,
\end{align}
for some $\epsilon>0$.
We note that (\ref{pss}) is just another way to express the condition
\begin{align}
S(\Omega):=\sup_{\lambda\in \Omega}\|(\lambda-L_0)^{-1}\|< \frac{1}{\epsilon}. \label{fin}
\end{align}
The $\eps$-pseudospectrum of linear operators has been studied extensively in the last two decades, both from an analytical and a numerical perspective, see the monographs \cite{MR2155029} and \cite{b_Davies} and references therein. 

 Our general result will provide estimates on the number $\mathcal{N}_L(\Omega')$ of discrete eigenvalues of $L$ (counting algebraic multiplicity) 
 in subsets $\Omega' \subset \Omega$.
We denote by $ \phi :\Omega \to \md $ a conformal mapping of $\Omega$ to the open unit disk $\md$, which satisfies $\phi(\infty)=0$, whose existence is assured by Riemann's Mapping Theorem. We define 
$$r_{\Omega}(\Omega'):=\sup_{z\in \Omega'}|\phi(z)|.$$ 
Note that $0 \leq r_{\Omega}(\Omega')\leq 1$, that $r_\Omega(\Omega')=0$ iff $\Omega'=\{\infty\}$, and that the values of  $r_\Omega$ do not depend on the choice of the conformal mapping $\phi$, since all such mappings differ only by a multiplicative constant of norm $1$. 

\begin{theorem}\label{thm1}
Let $p \in (0,\infty)$ and let $\Omega \subset \hat{\mc}$ be open and simply connected with $\infty \in \Omega$. Moreover, suppose that $\Omega$ satisfies (\ref{pss}) for some $\eps > 0$. Then for any $\Omega'\subset \Omega$ with $0<r_{\Omega}(\Omega')<1$ the following holds:  
  \begin{enumerate}
      \item[(i)] If $N\in\mathbb{N}_0$ is such that $\alpha_{N+1}(K)<\epsilon $, then 
\begin{equation*}
  \mathcal{N}_L(\Omega') \leq \frac{C_p}{(\epsilon-\alpha_{N+1}(K))^p \log \left(\frac 1 {r_{\Omega}(\Omega')}\right)} \sum_{j=1}^N\Big( \alpha_{N+1}(K) + \alpha_j(K) \Big)^p.
  \end{equation*}
\item[(ii)] If $\{\alpha_j(K)\} \in l^p(\mn)$, then 
\begin{equation*}
  \mathcal{N}_L(\Omega') \leq \frac{C_p}{\eps ^p \log \left(\frac 1 {r_{\Omega}(\Omega')}\right)}  \sum_{j=1}^\infty \alpha_j^p(K)  .
  \end{equation*}
  \end{enumerate}
In both cases, $C_p$ is as defined in (\ref{eq:9}).
\end{theorem}

\begin{proof}
Note that from Jensen's identity (see, e.g., \cite{rudin}) we know that for a bounded holomorphic function $h$ on $\md$ with $|h(0)|=1$ we have 
  \begin{equation*}
\int_0^1 \frac{n(h;s)}{s} ds \leq \log \|h\|_\infty,
  \end{equation*}
where $n(h;s)$ denotes the number of zeros of $h$ in $B_s$.
 From here we can deduce that for $0<r<1$
\begin{equation}\label{jl}n(h;r) \log \frac 1 r = \int_r^1 \frac{n(h;r)}{s} ds \leq \int_r^1 \frac{n(h;s)}{s} ds \leq \log \|h\|_\infty.\end{equation}
We will apply this result to the function $h=d \circ \phi^{-1}$, where $d : \Omega \to \mc$ is the holomorphic function from Theorem \ref{prop1}. Note
that by part (iii) of that theorem, every eigenvalue of $L$ in $\Omega'$ corresponds to a zero of $d$, hence to a zero of $h$ in $\phi(\Omega')$, which is a subset of the disk of radius 
$r=r_{\Omega}(\Omega')$ around $0$. Therefore 
\begin{equation}\label{inn1}
\mathcal{N}_L(\Omega') = \# \{ w \in \phi(\Omega') : h(w)=0\} \leq n(h,r_{\Omega}(\Omega')).
\end{equation}
By part (ii)  of Theorem \ref{prop1} and by (\ref{fin}), we have
$$\log \|h\|_\infty \leq  \frac{C_p}{(\epsilon-\alpha_{N+1}(K))^p}\cdot \sum_{j=1}^N\Big( \alpha_{N+1}(K) + \alpha_j(K) \Big)^p,$$
which together with (\ref{jl}) and (\ref{inn1}) implies (i). 

To obtain (ii)  from (i), we distinguish between the cases $0<p<1$ and $p \geq 1$, respectively. If $0<p<1$, we can use the inequality $(a+b)^p\leq a^p+b^p,$ ($a,b\geq 0)$ to obtain
\begin{equation}
  \label{eq:13}
 \sum_{j=1}^N\Big( \alpha_{N+1}(K) + \alpha_j(K) \Big)^p\leq \left( \sum_{j=1}^N \alpha_{N+1}^p(K) + \sum_{j=1}^N\alpha_j^p(K)\right). 
\end{equation}
In case that $p \geq 1$, the Minkowski inequality gives
\begin{small}
\begin{equation}
  \label{eq:14}
  \sum_{j=1}^N\Big( \alpha_{N+1}(K) + \alpha_j(K) \Big)^p  \leq \left[ \left( \sum_{j=1}^N \alpha_{N+1}^p(K) \right)^{1/p}+ \left( \sum_{j=1}^N\alpha_j^p(K)\right)^{1/p} \right]^p. 
\end{equation}
\end{small}
We note that since $j \mapsto \alpha_j(K)$ is non-increasing and $\{\alpha_j(K)\} \in l^p(\mn)$, we have 
$$\sum_{j=1}^N \alpha_{N+1}^p(K)= N\alpha_{N+1}^p(K) \to 0 \quad (N \to \infty) .$$ 
Indeed, we have
$$ (2j) \cdot \alpha_{2j}^p(K) =2 \sum_{m=j+1}^{2j} \alpha_{2j}^p(K) \leq 2\sum_{m=j+1}^{2j} \alpha_m^p(K) \to 0 \quad (j \to \infty),$$
and in a similar manner one can show that $(2j+1)\alpha_{2j+1}^p(K) \to 0$ for $j \to \infty$. Therefore, (\ref{eq:13}) and (\ref{eq:14}) imply that \emph{for all} $p>0$ 
$$\lim_{N\rightarrow \infty}\sum_{j=1}^N\Big( \alpha_{N+1}(K) + \alpha_j(K) \Big)^p\leq \sum_{j=1}^\infty\alpha_j^p(K),$$ 
so that taking $N\rightarrow \infty$ in (i) gives (ii).
\end{proof}

While the above result is very general, applying it to bound the number of eigenvalues in specific sets requires computing the quantity
$r_{\Omega}(\Omega')$, which is generally hard. We will here deal with the special but very important case of estimating the number of eigenvalues outside a disk, that
is we take $\|L_0\|<t<s$ and
$$\Omega=B_{t}^c,\;\;\;\Omega'= B_{s}^c.$$ 
Then $\Omega$ is simply connected, with $\infty \in \Omega$, and for $\lambda \in \Omega$ we have 
\begin{small}
\begin{eqnarray*}
  \|(\lambda-L_0)^{-1}\| = |\lambda|^{-1} \|(\mathds{1}-\lambda^{-1}L_0)^{-1}\| 
\leq |\lambda|^{-1} (1-|\lambda|^{-1}\|L_0\|)^{-1} < (t-\|L_0\|)^{-1},
\end{eqnarray*}
\end{small}
which shows that (\ref{pss}) holds with $\epsilon=t-\|L_0\|$.
The conformal mapping $\phi: \Omega\rightarrow \md$ is given by
$\phi(w)=\frac{t}{w},$ so that $r_{\Omega}(\Omega')=\frac{t}{s}$.
Therefore, denoting the number of eigenvalues (counted with multiplicities) of $L$ in $B_s^c$ by $n_L(s)$,  Theorem \ref{thm1} implies 
that if 
\begin{equation}\label{ulb}\|L_0\|+\alpha_{N+1}(K)<t<s\end{equation} 
then
	\begin{equation}\label{ppp0}
	n_L(s) \leq \frac{C_p  }{\log\left(\frac{s}{t}\right)[t-(\|L_0\|+\alpha_{N+1}(K))]^p}\sum_{j=1}^N\Big( \alpha_{N+1}(K) + \alpha_j(K) \Big)^p .
	\end{equation}
To optimize the bound we should take $t$ satisfying (\ref{ulb}) so as to minimize the right-hand side of (\ref{ppp0}). That is we need to 
maximize the function 
$$f(t)=\log\left(\frac{s}{t}\right)[t-a]^p,$$
where $$a=\|L_0\|+\alpha_{N+1}(K),$$
in the interval $(a,s)$ - note that this function vanishes at the endpoints and is positive in the interior of this interval, so that its
maximum is obtained in the interior. We can find it by setting $f'(t)=0$, where
$$f'(t)=-\frac{1}{t}[t-a]^p+p \log\left(\frac{s}{t}\right)[t-a]^{p-1}.$$

Denoting by $W(x)$ the Lambert W-function $W:[0,\infty)\rightarrow [0,\infty)$, which is defined by 
$$W(x)e^{W(x)}=x,$$
a short computation gives 
$$f'(t^*)=0\;\;\Leftrightarrow\;\; 1-\frac{a}{t^*}=p \log\left(\frac{s}{t^*}\right)\;\;\Leftrightarrow\;\;\frac{a}{pt^*}e^{\frac{a}{pt^*}}=\frac{a}{ps}e^{\frac{1}{p} }$$
$$\;\;\Leftrightarrow\;\;\frac{a}{pt^*}=W\left(\frac{a}{ps}e^{\frac{1}{p} }\right)\;\;\Leftrightarrow\;\;t^*=\frac{a}{pW\left(\frac{a}{ps}e^{\frac{1}{p} }\right)}.$$

Thus
$$\max_{t\in [a,s]}\log\left(\frac{s}{t}\right)[t-a]^p=f(t^*)=\log\left(\frac{ps}{a}W\left(\frac{a}{ps}e^{\frac{1}{p} }\right)\right)\left[\frac{a}{pW\left(\frac{a}{ps}e^{\frac{1}{p} }\right)}-a\right]^p $$ 
$$=\left[\frac{1}{pW\left(\frac{1}{p}e^{\frac{1}{p} }\cdot \frac{a}{s}\right)}-1\right]^{p+1}W\left(\frac{1}{p}e^{\frac{1}{p} }\cdot \frac{a}{s}\right) \cdot \frac{a^p}{s^p}\cdot s^p.$$
Therefore, defining $\Phi_p:(0,1)\rightarrow \mr$ by 
\begin{equation}
  \label{eq:11}
\Phi_p(x)=\frac{\left[W\left(\frac{1}{p}e^{\frac{1}{p} }x\right)\right]^p}{\left[\frac{1}{p}-W\left(\frac{1}{p}e^{\frac{1}{p} }x\right)\right]^{p+1}\cdot x^p},  
\end{equation}
we obtain the following result.
\begin{theorem}\label{thm:ball}
Let $p \in (0,\infty)$ and $s>\|L_0\|$. 
\begin{enumerate}
    \item[(i)] If $N\in\mathbb{N}_0$ is such that $\alpha_{N+1}(K)<s-\|L_0\|$, then
\begin{equation}\label{ppp}
n_L(s) \leq \frac{C_p}{s^p}\cdot \Phi_p\left(\frac{\|L_0\|+\alpha_{N+1}(K)}{s}\right)  \sum_{j=1}^N\Big( \alpha_{N+1}(K) + \alpha_{j}(K) \Big)^p.
  \end{equation}
 \item[(ii)] If $\{\alpha_j(K)\} \in l^p(\mn)$, then
\begin{equation}\label{qqq}
n_L(s) \leq \frac{C_p}{s^p} \cdot  \Phi_p\left(\frac{\|L_0\|}{s}\right)\sum_{j=1}^\infty\alpha_j^p(K).
  \end{equation}
\end{enumerate}
In both cases, $C_p$ is as given in (\ref{eq:9}). 
\end{theorem}
Here (ii) is obtained from (i) by taking $N\rightarrow \infty$, as in Theorem \ref{thm1}.

The previous theorem can be regarded as a broad generalization of the classical eigenvalue estimates for compact operators, as considered, e.g., in \cite{koenig} and \cite{pietsch}. 
Indeed, if $L$ is compact, i.e. $L_0=0$, we obtain from (\ref{qqq}) that  
\begin{equation} \label{eq:rst}
n_L(s) \leq \frac{p e C_p}{s^p} \cdot  \sum_{j}\alpha_j^p(L),
  \end{equation}
where we used the fact that, as a calculation shows, $\Phi_p(0)=\lim_{x\rightarrow 0}\Phi_p(x)=pe$. This inequality, 
 up to a constant, recovers the classical results. Estimate (\ref{ppp}) seems to be be new even in case that $L_0=0$. 

Concerning the asymptotic behavior (for $s \to \|L_0\|$) of the right-hand sides of (\ref{ppp}) and (\ref{qqq}), one can show that $\Phi_p(x) \sim (1-x)^{-(p+1)}$ for $x \to 1^-$, which, for instance,  in the summable case implies that  
\begin{equation}
  \label{eq:16}
  n_L(s) = O\left( \frac s {(s-\|L_0\|)^{p+1}} \right),  \text{ as } s \to \|L_0\|.
\end{equation}
The following corollary makes (\ref{eq:16}) more precise, and gives bounds on $n_L(s)$ which do not involve the function $\Phi_p$ and which are only slightly weaker than those of Theorem \ref{thm:ball}.
\begin{corollary}\label{corrr}
Let $p \in (0,\infty)$ and $s>\|L_0\|$. 
\begin{enumerate}
    \item[(i)] If $N\in\mathbb{N}_0$ is such that $\alpha_{N+1}(K)<s-\|L_0\|$, then
  \begin{small}
\begin{equation}
n_L(s) \leq \frac{C_p(p+1)^{p+1}}{p^p}\cdot \frac{s}{[s-(\|L_0\|+\alpha_{N+1}(K))]^{p+1}} \sum_{j=1}^N\Big( \alpha_{N+1}(K) + \alpha_{j}(K) \Big)^p.
  \end{equation}
  \end{small}
 \item[(ii)] If $\{\alpha_j(K)\} \in l^p(\mn)$, then
\begin{equation}\label{rrr}
n_L(s) \leq \frac{C_p(p+1)^{p+1}}{p^p}\cdot \frac{s}{(s-\|L_0\|)^{p+1}} \sum_{j=1}^\infty \alpha_j^p(K).
  \end{equation}
\end{enumerate}
In both cases, $C_p$ is as defined in (\ref{eq:9}).
\end{corollary}
 Note that (\ref{rrr}) is equivalent to the inequality (\ref{be}) presented in the Introduction (setting $C(p)=C_p\frac{(p+1)^{p+1}}{p^p}$).
\begin{proof}[Proof of Corollary \ref{corrr}] 
  The corollary is a direct consequence of Theorem \ref{thm:ball} and the estimate
  \begin{equation}
    \label{eq:15}
    \Phi_p(x) \leq \frac{(p+1)^{p+1}}{p^p}\cdot \frac{1}{(1-x)^{p+1}}, \quad 0<x<1.
  \end{equation}
To prove the last estimate, we define  
$$ g : (0,1) \to (0,1/p), \quad g(x)=W\left( \frac 1 p e^\frac 1 p x\right)$$ and 
$$h(x):= (1-x)^{p+1}\Phi_p(x)= p^{p+1} \left( \frac{g(x)}{x} \right)^p \left(\frac{1-x} {1-pg(x)} \right)^{p+1}, \quad x \in (0,1),$$
see (\ref{eq:11}). We show below that $h$ is monotonically increasing in $(0,1)$, so in particular
\begin{equation}
\label{eq:17}
h(x) \leq \lim_{y \to 1^-} h(y) = \frac{(p+1)^{p+1}}{p^p},
\end{equation}
where in the computation of the limit we used l'H\^opital's rule and the fact that $g(1)=1/p$. The validity of (\ref{eq:15}) is an immediate consequence of estimate (\ref{eq:17}).

To show that $h$ is monotonically increasing, we use the fact that $ W'(x)= \frac 1 x \cdot \frac {W(x)}{W(x)+1}$,
and so  
$$ g'(x)= \frac 1 x \cdot \frac {g(x)}{g(x)+1},$$
and differentiate $h(x)$, obtaining 
$$h'(x)= p^{p+1}\cdot \frac{(1-x)^p g(x)^p}{x^{p+1}(1-pg(x))^{p+2}}\cdot \Big( p (x+p) g(x)-(p+1)x \Big).$$
Thus we have $h'(x)>0$ for all $x \in (0,1)$ if and only if $$f(x):= \frac{p}{p+1}(x+p)-\frac x {g(x)} > 0$$ for all $x \in (0,1)$. But $\lim_{x \to 1^-}f(x)=0$, and $f$ is strictly monotonically decreasing since $f'(x)= \frac p {p+1} - \frac 1 {g(x)+1} < 0$ for all $x \in (0,1)$. Therefore $f(x) > 0$ and hence $h'(x) > 0$ for all $x \in (0,1)$.
\end{proof}

\section{Remarks on the sharpness of the results}

We now express our results in terms of bounds on sums of powers of eigenvalues of $L=L_0+K$ outside the disk of radius $\|L_0\|$. Besides the 
intrinsic interest in such a formulation, it will be convenient for discussing issues related to the sharpness of the results obtained. 

By integration by parts one has 
\begin{equation}
\label{eq:1}
q \int_{\|L_0\|}^\infty n_L(s) (s-\|L_0\|)^{q-1} ds = \sum_{\lambda \in \sigma_d(L), |\lambda|>\|L_0\|}  (|\lambda|-\|L_0\|)^q, \quad q>0,  
\end{equation}
where in the sum each eigenvalue is counted according to its algebraic multiplicity. 

Using (\ref{rrr}), and the fact that $n_L(s)=0$ for $s>\|L_0\|+\|K\|\geq \|L\|$, we obtain,  
\begin{eqnarray}&&\label{dh} \sum_{\lambda \in \sigma_d(L), |\lambda|>\|L_0\|}  (|\lambda|-\|L_0\|)^q  = q \int_{\|L_0\|}^\infty n_L(s) (s-\|L_0\|)^{q-1} ds\\
&\leq& q \frac{C_p(p+1)^{p+1}}{p^p} \sum_{j}\alpha_j^p(K)\int_{\|L_0\|}^{\|L_0\|+\|K\|}  \frac{s }{(s-\|L_0\|)^{p+2-q}}   ds\nonumber\\
&=&q \frac{C_p(p+1)^{p+1}}{p^p}\left[\frac{1}{q-p-1}\|L_0\|+ \frac{1}{q-p} \|K\| \right]\|K\|^{q-p-1}\cdot\sum_{j}\alpha_j^p(K),\nonumber
\end{eqnarray}
where the finiteness of the integral, hence the validity of the inequality, requires $q>p+1$ if $L_0\neq 0$, and $q>p$ if $L_0=0$.

We thus have the following facts, where we distinguish between the cases $L_0=0$ (which implies that $L$ is compact) and $L_0\neq 0$. 
\begin{corollary}\label{cor1}
Let $L_0, K \in \bdd(X)$ and $L:=L_0+K$. 
\begin{enumerate} 
    \item[(i)] If $L_0 \neq 0$, then for any $p>0$,  $q>p+1$
\begin{equation}
  \label{eq:3}
 \{\alpha_j(K)\} \in l^p(\mn)\;\;\; \Rightarrow \sum_{\lambda \in \sigma_d(L), |\lambda|>\|L_0\|}  (|\lambda|-\|L_0\|)^{q}<\infty.
\end{equation}
    \item[(ii)] If $L_0 = 0$, then for every $p>0$, $q>p$  
\begin{equation}
  \label{eq:33}
 \{\alpha_j(K)\} \in l^p(\mn) \;\;\;\Rightarrow \sum_{\lambda \in \sigma_d(L)}  |\lambda|^{q} <\infty .
\end{equation}
\end{enumerate}
\end{corollary}

Noting the difference in the condition on the exponent $q$ between the cases $L_0\neq 0$ ($q>p+1$) and the case $L_0=0$ ($q>p$), it is natural 
to ask whether this reflects a real difference in the possible distribution of eigenvalues in the two cases, or a limitation of our 
methods of proof. That is, we seek to determine to what extent the results we have obtained are sharp.  We therefore ask:  
\begin{enumerate} 
    \item What is the infimum $q_{B}(p)$ $[q_{{B}}^0(p)]$ of all exponents $q$ such that the implication (\ref{eq:3}) [(\ref{eq:33})] is valid for all Banach spaces $X$, all $L_0 \in \bdd(X)$ and all compact $K \in \bdd(X)$ with $\{\alpha_j(K)\}\in\l^p(\mn)$?
  \item Is the above infimum a minimum?
\end{enumerate}
Let us first consider the case $L_0=0$: Here, it is well-known that for all $p>0$ we have $q_{{B}}^{(0)}(p)=p$ and that the infimum is a minimum, as follows from K\"onig's result (\ref{eq:7}) above. So we see that for this case we almost recover the optimal exponent. 
 
What about the case $L_0\neq 0$? Our results imply 
$$ \max(1,p) \leq  q_{{B}}(p) \leq  p+1,$$
where the lower bound follows from (\ref{hilbert}) below, while the upper bound follows from Corollary \ref{cor1}.  Otherwise we do not know 
much about the value $q_{{B}}(p)$, nor whether the infimum is a minimum. 

Let us note that if we restrict ourselves to Hilbert spaces, and define the constant $q_{H}(p)$ analogously, then it is known that
\begin{equation}
\label{hilbert}
 q_{{H}}(p) = \max(1,p)
\end{equation}
and that the infimum is again a minimum, as follows from results in \cite{Hansmann11} and \cite{MR3033958}. Thus for the case of general $L_0,L$ on a Hilbert space the situation is the same as for $L_0=0$ on a general Banach space as long as $p\geq 1$, but quite different for $p$ smaller than one.

The question is thus whether the fact that the results concerning the exponent $q_{B}(p)$ that we obtain are weaker than the known
results for Hilbert space operators is due to non-sharpness of our results, or rather to a real difference between what can happen in Hilbert spaces
and in general Banach spaces, respectively. If the latter is the case, then this must be demonstrated by constructing appropriate examples. 
While we do not have an answer to the above question, we do have an example which shows that eigenvalues of perturbations can  behave
in a `worse' way in general Banach spaces than in Hilbert spaces.
 Indeed, below we will construct an example with $X=l^1(\mn)$, where $L-L_0$ is of {\it{finite rank}}, and where
$$ \sum_{\lambda \in \sigma_d(L), |\lambda|>\|L_0\|}  (|\lambda|-\|L_0\|)=\infty.$$
Note that for a finite rank perturbation on a {\it{Hilbert}} space the above sum will always be finite, as follows from the considerations above and the fact that the approximation numbers of finite rank operators are $p$-summable for every $p>0$.

\begin{example}
It is well-known (see, e.g., \cite{MR0307393}) that there exist holomorphic functions $h$ on the unit disk, with uniformly bounded Taylor coefficients, such that
\begin{equation}
  \label{eq:10}
\sum_{w \in \md, h(w)=0} (1-|w|) = \infty,  
\end{equation}
where each zero is counted according to its order. Let us fix such a (normalized) function
$$ h(w)=1-\sum_{k=1}^\infty b_k w^k,$$
with $\{ b_k\} \in l^\infty(\mn)$. Now we choose $X=l^1(\mn)$ and let $L_0$ denote the shift operator on $l^1(\mn)$, i.e. 
$$ L_0 \delta_n= \delta_{n+1}, \quad n \in \mn,$$
where $\{\delta_n\}$ denotes the canonical Schauder basis of $l^1(\mn)$.
Clearly, $\|L_0\|=1$. Next, we define a rank one operator $K$ on $l^1(\mn)$ by
$$ Kf = \langle f, b \rangle \delta_1,$$
where $\langle ., . \rangle$ denotes the dual pairing between $l^1$ and $l^\infty$, and we set $L=L_0+K$. For $|\lambda| > 1$ we then have that $\lambda \in \sigma_d(L)$ iff
$$ {\det}_1(\mathds{1}-K(\lambda-L_0)^{-1})=0.$$
It is not difficult to see that, setting $w=\lambda^{-1}$,
\begin{eqnarray*}
{\det}_1(\mathds{1}-K(\lambda-L_0)^{-1}) &=& 1- \langle (\lambda-L_0)^{-1}\delta_1,b \rangle \\
&=& 1- w \sum_{k=0}^\infty \langle L_0^k\delta_1,b \rangle w^k \\
&=& 1- \sum_{k=1}^\infty b_{k} w^k = h(w).
\end{eqnarray*}
From (\ref{eq:10}) we thus obtain that
$$ \sum_{\lambda \in \sigma_d(L), |\lambda|>\|L_0\|} (|\lambda|-\|L_0\|)= \sum_{w \in \md, h(w)=0} \frac{1-|w|}{|w|} = \infty.$$
\end{example} 
We thus conclude this article with a number of open problems: Is it true that for general Banach spaces we have that $q_{{B}}(p)=p$ for $p \geq 1$ (but maybe it will not be a minimum)? Or do we have that $q_{{B}}(p)$ is strictly larger than $p$? Is our upper bound $p+1$ actually equal to $q_{{B}}(p)$? In addition, one might also ask about the optimal exponents for more specific classes of Banach spaces (only the case of Hilbert spaces being known). These are, we believe, intriguing questions for further investigation.

\end{document}